\documentclass[12pt,letterpaper]{amsart}
\usepackage{amsmath,amsfonts,amssymb, tikz, amsthm}
\usepackage[colorlinks=true,linkcolor=black,anchorcolor=black,citecolor=black,filecolor=black,menucolor=black,runcolor=black,urlcolor=blue]{hyperref}
\usepackage[utf8]{inputenc}
\usepackage[english]{babel}
\usepackage{xcolor}
\usepackage[tableposition=top]{caption} 
\usepackage{epigraph}
\usepackage{mathtools}

\newtheorem*{theorem*}{Theorem}
\newtheorem{theorem}{Theorem}
\newtheorem{lemma}{Lemma}

\newtheorem{corollary}{Corollary}

\newtheorem{definition}{Definition}

\theoremstyle{remark}
\newtheorem{example}{Example}
\newtheorem{remark}{Remark}

\numberwithin{equation}{section}
\numberwithin{assumption}{section}

\Alph{assumption}

\newcommand{\Z}{\mathbb{Z}}
\newcommand{\Q}{\mathbb{Q}}

\newcommand{\N}{\mathbb{N}}

\begin{document}
\title{Diophantine tuples and Integral Ideals of $\mathbb{Q}(\sqrt{d})$}
	
	\author{Kalyan Chakraborty}
	\email[Kalyan Chakraborty]{kalyan.c@srmap.edu.in, kalychak@gmail.com}
	\address{Department of Mathematics, SRM University AP
		Neerukonda, Mangalagiri, Guntur-522240
		Andhra Pradesh
		India.}
	
	\author{Shubham Gupta}
	\email[Shubham Gupta]{shubhamgupta2587@gmail.com}
	\address{Harish-Chandra Research Institute,  A CI of Homi Bhabha National
		Institute, Chhatnag Road, Jhunsi, Prayagraj - 211019, India.}
	
	\author{Krishnarjun Krishnamoorthy}
	\email[Krishnarjun Krishnamoorthy]{krishnarjunmaths@outlook.com}
	\address{Beijing Institute of Mathematical Sciences and Applications (BIMSA), No. 544, Hefangkou Village, Huaibei Town, Huairou District, Beijing.}
	
	\keywords{Diophantine tuples, Quadratic extensions}
	\subjclass[2020]{11D09, 11R11}

	\begin{abstract}
		Suppose $n$ is the fundamental discriminant associated with a quadratic extension of $\Q$. We show that for every Diophantine $m$-tuple $ \{t_1, t_2, \ldots, t_m\} $ with the property $ D(n) $, there exists integral ideals $ \mathfrak{t}_1, \mathfrak{t}_2, \ldots, \mathfrak{t}_m $ of $ \Q(\sqrt{n}) $ and $c\in \{1,2\}$ such that $ t_i= c\mathcal{N}(\mathfrak{t}_i) $ for $ i=1,2, \ldots, m $. Here, $ \mathcal{N}(\cdot) $ denotes the norm map from $\Q(\sqrt{n})$ to $\Q$. Moreover, we explicitly construct the above ideals for Diophantine pairs $\{a_1, a_2\}$ whenever $\gcd(a_1, a_2) = 1$.
	\end{abstract}
	
	\maketitle
	
	\section{Introduction}\label{Section "Introduction"}
	
	Solving Diophantine equations is an area of mathematics, particularly number theory with a long and rich history. The focus of this article is a particular family of simultaneous Diophantine equations parameterized by certain ``moduli''. More precisely, given an integer $ n $, we say that a set of $m$ distinct elements $ \{a_1, a_2,\ldots, a_m\} \subset \N$ is a Diophantine $ m $-tuple with the property $ D(n) $, or a $D(n)$-$m$-tuple, if for all $ 1\leq i < j \leq m $, $ a_ia_j+m $ is a perfect square in $\mathbb{Z}$. One of the first examples of such sets goes back to Fermat who gave the example $ \{1,3,8,120\} $ which is a $ D(1) $-quadruple. Recently, Diophantine tuples are being extensively studied (see for example \cite{ DujeCrelle,TogbeZiegler, BCC2022, CMF2024, VZ2021}) and they are connected to diverse areas of number theory. Readers may also see books (\cite{Dujebook, DUJE_1}) written by Dujella to get the more information about Diophantine $m$-tuples.
	
	The aim of this article is to connect Diophantine $ m $-tuples to the algebraic study of quadratic extensions of $ \Q $. Most of the analysis seems to be generalizable for other base fields, particularly number fields, but we do not pursue that here. Before we proceed to state our results, we first briefly describe the various techniques available today in literature for studying Diophantine $ m $-tuples.
	
	One of the fundamental questions in the theory of Diophantine tuples is the question of extension of triples to quadruples, quintuples etc. Similarly, one might ask for a bound on the largest possible $ m $ for which there exists a Diophantine $ m $-tuple with the property $ D(n) $ for a given $ n $. The starting point for the solution of these equations almost always involves the solution of a simultaneous system of generalized Pell's equations. One then proceeds to deduce bounds on the possible solutions using the theory of linear forms of logarithms and hopes that these bounds would lead to some information about the existence of Diophantine tuples as required (see \cite{DujeCrelle,TogbeZiegler, BCC2022}).
	
	Another conceptual way of dealing with this problem is to associate algebro-geometric objects to Diophantine tuples and ask for the existence of rational or even integral points on certain varieties. In particular, given a $ D(1) $-triple $ \{a,b,c\} $, one may associate an elliptic curve and obtain a correspondence between the integral points of that particular elliptic curve and the number of possible extensions of $ \{a,b,c\} $ to a $ D(1) $-quadruple (see \cite{FUJ2007, DK2017}). The advantage of this method is that there are some remarkable results regarding the existence of integral points on elliptic curves which one may then employ to deduce theorems about Diophantine tuples. The downside is that the theoretical results that exist are not effective and the bounds arising from their effective generalizations are too large to be of any practical use.
	
	 To the best of our knowledge, the connection between Diophantine tuples and the norm map has not been explored and we hope that this leads to new directions in the study of Diophantine tuples. The fundamental lemma is Lemma \ref{Lemma "Diophantine - quadratic"} which connects Diophantine pairs to certain special algebraic integers in a quadratic extension of $ \Q $. This immediately allows us to determine the divisibility properties of Diophantine pairs with respect to various primes. This is described in Lemma \ref{Lemma "Prime power divisibility"}. The statements and proofs of the above lemmas constitute \S \ref{Section "Preliminary lemmas"}. In order to state the main result of this paper, we make the following definition.
	
	\begin{definition}\label{Definition "Norm Tuple and Principal"}
		A Diophantine $m$-tuple with the property $ D(n) $ of the form $ \{\mathcal{N}(\mathfrak{a}_1), \ldots,  \\ \mathcal{N}(\mathfrak{a}_m)\} $, for integral ideals $ \mathfrak{a}_1, \ldots, \mathfrak{a}_m $ of $ \Q(\sqrt{n}) $ is called a norm tuple.   If all the integral ideals $ \mathfrak{a}_1, \ldots, \mathfrak{a}_m  $ are principal, then we shall call the tuple as a principal norm tuple.
	\end{definition}
	
	While considering a norm or principal norm Diophantine $m$-tuple with the property $D(n)$, it is essential to consider the particular value of $ n $. However, to not clutter the notation we shall not do so, as the choice of $ n $ will be clear from the context.
	
    \subsection{Main Results}
    
    In this paper, we prove the following two theorems. The first theorem is related to the existence of an integral ideal of $\mathbb{Q}(\sqrt{n})$ for a Diophantine $m$-tuple with the property $D(n)$. More precisely, we have the following result:

	\begin{theorem}\label{Theorem "Every tuple is a norm tuple"}
		Suppose that $ n $ is a fundamental discriminant and $ T:=\{t_1, t_2,\ldots, t_m\} $ is a Diophantine $ m $-tuple with the property $ D(n) $. Then either $ T $ is a norm tuple, or $ T = 2T' $ where $ T' $ is a norm tuple with the property $ D(n/4) $. That is to say all the $ t_i $s are divisible by $ 2 $ and the tuple $ T':= \{t_1/2, t_2/2, \ldots, t_m/2\} $ is a norm tuple.
	\end{theorem}
	
	The proof of Theorem \ref{Theorem "Every tuple is a norm tuple"} is broken down into parts and is given in \S \ref{Section "Norm Tupls and Principal norm tuples"}. With some effort minor generalizations of Theorem \ref{Theorem "Every tuple is a norm tuple"} maybe possible by relaxing the requirement that $ n $ be the fundamental discriminant  and in particular we may allow for non-square-free values of $ n $. 
	
    The proof of the above theorem is existential and it would be very convenient to be able to explicitly construct the ideals. Under some additional restrictions, we explicitly construct these ideals in the following theorem. Given two elements $\alpha, \beta \in \mathbb{Q}(\sqrt{n})$, we denote the integral ideal generated by them as $\langle \alpha, \beta\rangle$.

    \begin{theorem}\label{Theorem "ideal"}
    	Suppose that $n$ is a square free rational integer and $x \in \mathbb{Z}$. Let $a_1, a_2$ be two co-prime natural numbers such that $a_1a_2 + n = x^2$ for some $x\in \Z$. Then for $i = 1, 2$, we may choose  $\mathfrak{a}_i = \langle a_i, x + \sqrt{n}\rangle$ in Theorem \ref{Theorem "Every tuple is a norm tuple"}, i.e. $\mathcal{N}(\mathfrak{a}_i) = a_i$ for $i=1,2$. Furthermore $\mathfrak{a}_1 \mathfrak{a}_2$ is a principal integral ideal. 
    \end{theorem}

    \begin{remark}
        We highlight that the construction of the ideals mentioned above depends on the pair $\{a_1, a_2\}$ and not just the integer $a_1$. This is relevant because if we start with a triple $\{a_1, a_2, a_3\}$, the ideal asociated to $\{a_1,a_3\}$, say $\mathfrak{a}_1$ with $\mathcal{N}(\mathfrak{a}_1)=a_1$ may be different from the corresponding ideal $\mathfrak{a}_1'$ say associated with the pair $\{a_1,a_3\}$.
    \end{remark}

    \subsection*{Notations and conventions}
	\begin{enumerate}
		\item  	We shall denote the fundamental discriminant of $ \Q(\sqrt{n}) $ as $ D_n $. We denote the corresponding trace and norm maps as $\mathcal{T}(\cdot)$ and $\mathcal{N}(\cdot)$, respectively.
        \item 	Let $p$ be a prime number and $a \in \mathbb{N}$. The notation $\left(\dfrac{a}{p} \right)$ denotes the Legendre symbol.
	\end{enumerate}

    \subsection*{Acknowledgments}

    The authors thank Preda Mihailescu for useful discussions.

	\section{Preliminary Lemmas}\label{Section "Preliminary lemmas"}
	
	\begin{lemma}\label{Lemma "Diophantine - quadratic"}
		For rational integers $ t_1, t_2 $ and $ n $, $ t_1t_2 + n = r^2 $ if and only if there exists an algebraic integer $ \alpha $ of the form $ a + \sqrt{n} $ such that $ \mathcal{T}(\alpha) = -2r $ and $ \mathcal{N}(\alpha) = t_1t_2 $.
	\end{lemma}
	
	\begin{proof}
		If $ t_1t_2 + n = r^2 $, then we may define $ \alpha $ to be a root of the polynomial $ x^2 + 2rx + t_1t_2 $. Conversely, if such an $ \alpha $ exists, it follows that $ \mathcal{N}(\alpha) = t_1t_2 = r^2-n $. Rewriting this equality, we see that $ t_1t_2 + n = r^2 $ proving the lemma.
	\end{proof}
	
	\begin{lemma}\label{Lemma "Prime power divisibility"}
		With notation as in Lemma \ref{Lemma "Diophantine - quadratic"}, for any prime $ q $ dividing $ t_1t_2 $, $ q^{f_q} $ divides $ t_1t_2 $, where $ f_q $ is the residue degree of $ q $ in the extension $ \Q(\sqrt{n}) $.
	\end{lemma}
	
	\begin{proof}
		From Lemma \ref{Lemma "Diophantine - quadratic"}, there exists an $ \alpha $ such that $ \mathcal{N}(\alpha) = t_1t_2 $. If $ q | t_1t_2 $, and if $ \mathfrak{q} $ is a prime ideal above $ q $, then either $ \mathfrak{q} | (\alpha) $ or $ \overline{\mathfrak{q}} | (\alpha) $. In both cases, $ \mathcal{N}(\mathfrak{q}) = q^{f_q} | \mathcal{N}(\alpha) = t_1t_2 $, whence the claim.
	\end{proof}
	
	\section{Proof of Theorem \ref{Theorem "Every tuple is a norm tuple"}.}\label{Section "Norm Tupls and Principal norm tuples"}
	
	\begin{lemma}\label{Lemma "Norm Forms"}
		Every Diophantine $m$-tuple with the property $D(n)$ is a norm tuple or a constant times a norm tuple. Moreover the constant maybe chosen to be outside the image of the norm map.
	\end{lemma}
	
	\begin{proof}
		Suppose that $ \{t_1, t_2\} $ is a $ D(n) $-pair which is not a norm pair. Then there is a prime $ q $ which is inert in $ \Q(\sqrt{n}) $ such that $ q^a $ properly divides $ t_1 $ for some odd integer $ a $. But from Lemma \ref{Lemma "Diophantine - quadratic"}, $t_1t_2$ is contained in the image of the norm map from $\Q(\sqrt{n})$ to $\Q$ and hence $ q | t_2 $. Thus if we define $ \kappa $ as the product of all the inert primes properly dividing $ t_1 $ with an odd exponent, then $ \kappa | t_2 $. Moreover, $ t_1/\kappa $ and $ t_2/\kappa $ are contained in the image of the norm map. Furthermore, $ \{t_1/\kappa, t_2/\kappa\} $ is a $ D(n/\kappa^2) $-pair. Whence the lemma.
	\end{proof}
	
	\begin{remark}
		When we rewrite a Diophantine $m$-tuple as a constant times a norm tuple, observe that the modulus becomes different. In fact, it is not even necessary that $ n $ is divisible by $ \kappa^2 $.
	\end{remark}
	
	The following corollaries are self-evident.
	
	\begin{corollary}\label{Corollary "Norm Tuple"}
    		A Diophantine $m$-tuple $\{a_1,\ldots, a_m\}$ with the property $D(n)$  having \\
            $\gcd (a_1,\ldots, a_m) = 1$ is a norm tuple.
	\end{corollary}
	
	\begin{corollary}
		A Diophantine $m$-tuple which contains $ 1 $ is a norm tuple.
	\end{corollary}
	
	
	\begin{lemma}\label{Lemma "Every tuple is norm"}
		If $ \{t_1, t_2\} $ is a $ D(n) $-pair and if $ p $ is an odd prime dividing $ t_1t_2 $ and co-prime to $ n $, then $ p $ is split in $ \Q(\sqrt{n}) $.
	\end{lemma}
	
	\begin{proof}

	Suppose the contrary that $ p $ is inert in $ \Q(\sqrt{n}) $ and that $ p | t_1t_2 $.	Let us denote the discriminant of $ \Q(\sqrt{n}) $ as $ D_n $. Then from \cite[Theorem 1.30]{RAM2011}, we have
		\begin{equation}\label{eq1}
			\left(\frac{D_n}{p}\right) = -1.
		\end{equation} 
		
    Suppose that $t_1t_2 + n = r^2$ for some $r \in \mathbb{Z}$. In particular, as $ p | t_1t_2 $, we may view the above equation modulo $ p $ which gives us $ n \equiv r^2 \mod p $. Furthermore, as $ D_n = s^2 n $ for some rational number $ s $, it follows that $ D_n $ is a quadratic residue modulo $ p $. 
		In other words, 
		\begin{equation}\label{eq2}
		\left(\frac{D_n}{p}\right) = 1.
		\end{equation}
		From \eqref{eq1} and \eqref{eq2}, we get the contradiction. Hence $p$ is not inert in $\mathbb{Q}(\sqrt{n})$. 
        
        By hypotheses, $p$ is an odd prime having $\gcd(p, n) = 1$, so $p$ is not ramified in $\mathbb{Q}(\sqrt{n})$. Thus, $p$ is split in $\mathbb{Q}(\sqrt{n})$ whence the lemma.   
	\end{proof}
	
	\begin{remark}
		The reason that we require $ p $ to be co-prime to $ n $ is to weed out examples of the following form. Suppose that $ \{t_1,t_2\} $ is a $ D(n) $-pair and that $ p $ is a prime inert in $ \Q(\sqrt{n}) $. Clearly, then, $ \{pt_1,pt_2\} $ is a $ D(np^2) $-pair, but $ \Q(\sqrt{n}) = \Q(\sqrt{np^2}) $. Thus in this situation, we may trivially construct counter examples. The case of the prime $ 2 $ is somewhat special as the following example illustrates.
	\end{remark}
	
	\begin{example}
		It is worth mentioning that Lemma \ref{Lemma "Every tuple is norm"} is not valid when we consider the case of the prime $ 2 $. We observe that $ \{2,6\} $ is a $ D(-3) $-pair but as $ 2 $ is inert in $ \Q(\sqrt{-3}) $, neither $ 2 $ nor $ 6 $ is a norm of an integral ideal in $ \Q(\sqrt{-3}) $. Similarly we may easily verify that $ \{2,10\} $ and $ \{2,6,18\} $ are $ D(5) $-pair and $ D(13) $-triple,  respectively.
	\end{example}
	
	\begin{corollary}\label{Corollary "Norm tuples fundamental discriminant"}
		Suppose that $ n $ is a fundamental discriminant, then in Lemma \ref{Lemma "Norm Forms"}, we may choose $ \kappa $ to be either $ 1 $ or $ 2 $. The choice $ \kappa = 2 $ is possible only if $ n $ is odd and $ 2 $ is inert in $ \Q(\sqrt{n}) $.
	\end{corollary}
	
	\begin{proof}
		With notation as in the proof of Lemma \ref{Lemma "Norm Forms"}, $ \kappa $ is the product of all the inert primes that divide $ t_1$ and $t_2 $. As $ n $ is chosen to be a fundamental discriminant, none of the inert primes $ p $ divides $ n $. But from Lemma \ref{Lemma "Every tuple is norm"}, the only such prime, if it exists, is $ 2 $, and this choice is possible only if $ 2 $ is inert in $ \Q(\sqrt{n}) $. This completes the proof.
	\end{proof}
	
	\begin{proof}[Proof of Theorem \ref{Theorem "Every tuple is a norm tuple"}]
		Theorem \ref{Theorem "Every tuple is a norm tuple"} is now evident from Corollary \ref{Corollary "Norm tuples fundamental discriminant"}.
	\end{proof}
	
	\begin{corollary}
		If $ n $ is a fundamental discriminant such that $ 2 $ is not inert in $ \Q(\sqrt{n}) $, then every $ D(n) $-tuple is a norm tuple.
	\end{corollary}

    \section{Proof of Theorem \ref{Theorem "ideal"}}\label{ex_id}

	Let notation be as in the theorem. For $i = 1, 2$, let $\mathfrak{a}_i = \langle a_i,~ x + \sqrt{n}\rangle$. Clearly
	\[
        \mathcal{N}(\mathfrak{a}_i) ~|~ \mathcal{N}(a_i) = a_i^2,~~ \mathcal{N}(\mathfrak{a}_i) ~|~ \mathcal{N}(x + \sqrt{n}) = x^2-n = a_1a_2.
        \]
	From the co-primality of $a_1, a_2$, we get that
	\begin{equation}\label{NAi}
	    \mathcal{N}(\mathfrak{a}_i) | a_i.
	\end{equation}
	Furthermore,
	\[
    \mathfrak{a}_1\mathfrak{a}_2 = \langle a_1, x+ \sqrt{n}\rangle\langle a_2, x+\sqrt{n}\rangle = \langle a_1a_2, a_1(x + \sqrt{n}), a_2(x + \sqrt{n}), (x + \sqrt{n})^2\rangle \subset \langle x + \sqrt{n}\rangle.
    \]
	Since by hypothesis $\gcd(a_1, a_2) = 1$, there exist $y, z \in \mathbb{Z}$ such that $a_1x + a_2y = 1$. So, $x + \sqrt{n} \in \mathfrak{a}_1\mathfrak{a}_2$ giving us $\mathfrak{a}_1\mathfrak{a}_2 = \langle x + \sqrt{n}\rangle$.
	This gives us $\mathcal{N}(\mathfrak{a}_1\mathfrak{a}_2) = x^2 - n = a_1a_2$.
	Utilizing \eqref{NAi} along with the last equation, we conclude that $\mathcal{N}(\mathfrak{a}_i) = a_i$	for $i = 1, 2$ completing the proof.

	\section{Concluding Remarks}
	
	Theorem \ref{Theorem "Every tuple is a norm tuple"} associates Diophantine tuples with the image of the norm map from quadratic number fields. In order to consider higher degree extensions, we need to consider a different family of Diophantine equations which go beyond the classical definition of Diophantine $m$-tuples. They are as follows:  A set  $\{a_1,\ldots, a_m\}$ is called a Diophantine $m$-tuple with the property $D_k(n)$ if it has the property  $a_ia_j + n = x^k$, where $x \in \mathbb{Z}$ and $k \geqslant 2$.
	 Till now, there are very few results (see \cite{DKM2022} and \cite{YD2003}) and examples for Diophantine $m$-tuple with the property $D_k(n)$. 
	For example, $\{2,171,25326\}$ and $\{1352,9539880,9768370\}$ are Diophantine triple with the property $D_3(1)$ and $D_4(1)$, respectively. As far as our knowledge, no example of such triple with the property $D_k(1)$ is known for $k \geq 5.$

	\bibliographystyle{amsalpha} 
	
	\bibliography{bibliography1}

\providecommand{\bysame}{\leavevmode\hbox to3em{\hrulefill}\thinspace}
\providecommand{\MR}{\relax\ifhmode\unskip\space\fi MR }
\providecommand{\MRhref}[2]{%
  \href{http://www.ams.org/mathscinet-getitem?mr=#1}{#2}
}
\providecommand{\href}[2]{#2}
\begin{thebibliography}{DKMS17}

\bibitem[BCM22]{BCC2022}
Nicolae~Ciprian Bonciocat, Mihai Cipu, and Maurice Mignotte, \emph{There is no
  {D}iophantine {$D(-1)$}-quadruple}, J. Lond. Math. Soc. (2) \textbf{105}
  (2022), no.~1, 63--99. \MR{4411320}

\bibitem[BD03]{YD2003}
Yann Bugeaud and Andrej Dujella, \emph{On a problem of {D}iophantus for higher
  powers}, Math. Proc. Cambridge Philos. Soc. \textbf{135} (2003), no.~1,
  1--10. \MR{1990827}

\bibitem[CF24]{CMF2024}
Mihai Cipu and Yasutsugu Fujita, \emph{On the length of {$D(\pm1)$}-tuples in
  imaginary quadratic rings}, Bull. Lond. Math. Soc. \textbf{56} (2024), no.~1,
  274--287. \MR{4705323}

\bibitem[DKM22]{DKM2022}
Anup~B. Dixit, Seoyoung Kim, and M.~Ram Murty, \emph{Generalized {D}iophantine
  {$m$}-tuples}, Proc. Amer. Math. Soc. \textbf{150} (2022), no.~4, 1455--1465.
  \MR{4375736}

\bibitem[DKMS17]{DK2017}
Andrej Dujella, Matija Kazalicki, Miljen Miki\'{c}, and M\'{a}rton Szikszai,
  \emph{There are infinitely many rational {D}iophantine sextuples}, Int. Math.
  Res. Not. IMRN (2017), no.~2, 490--508. \MR{3658142}

\bibitem[Duj04]{DujeCrelle}
Andrej Dujella, \emph{There are only finitely many {D}iophantine quintuples},
  J. Reine Angew. Math. \textbf{566} (2004), 183--214. \MR{2039327}

\bibitem[Duj21]{Dujebook}
Andrej Dujella, \emph{Number theory}, \v{S}kolska Knjiga, Zagreb, 2021,
  Translated from the Croatian edition by Petra \v{S}vob. \MR{4368213}

\bibitem[Duj24]{DUJE_1}
\bysame, \emph{Diophantine {$m$}-tuples and elliptic curves}, Developments in
  Mathematics, vol.~79, Springer, Cham, [2024] \copyright 2024. \MR{4769927}

\bibitem[Fuj07]{FUJ2007}
Yasutsugu Fujita, \emph{The {$D(1)$}-extensions of {$D(-1)$}-triples
  {$\{1,2,c\}$} and integer points on the attached elliptic curves}, Acta
  Arith. \textbf{128} (2007), no.~4, 349--375. \MR{2320718}

\bibitem[HTZ19]{TogbeZiegler}
Bo~He, Alain Togb\'{e}, and Volker Ziegler, \emph{There is no {D}iophantine
  quintuple}, Trans. Amer. Math. Soc. \textbf{371} (2019), no.~9, 6665--6709.
  \MR{3937341}

\bibitem[Mol11]{RAM2011}
Richard~A. Mollin, \emph{Algebraic number theory}, second ed., Discrete
  Mathematics and its Applications (Boca Raton), CRC Press, Boca Raton, FL,
  2011. \MR{2779314}

\bibitem[Zie21]{VZ2021}
Volker Ziegler, \emph{Finding all {$S$}-{D}iophantine quadruples for a fixed
  set of primes {$S$}}, Monatsh. Math. \textbf{196} (2021), no.~3, 617--641.
  \MR{4320541}

\end{thebibliography}

\end{document}